\documentclass[11pt,a4paper]{article}
\usepackage[T1]{fontenc}
\usepackage[utf8]{inputenc}
\usepackage{lmodern}
\usepackage{amsmath,amssymb,amsthm}
\usepackage{microtype}
\usepackage[margin=28mm,headheight=14pt]{geometry}
\usepackage{booktabs}
\usepackage{xcolor}
\usepackage{xurl}
\usepackage{hyperref}
\usepackage{fancyhdr}
\hypersetup{colorlinks=true,linkcolor=blue!45!black,citecolor=blue!45!black,
  urlcolor=blue!45!black,pdftitle={An explicit counterexample to the Hinrichs--Vybiral conjecture},
  pdfsubject={Nonnegative positive definite functions; an exact counterexample in dimension seven}}
\allowdisplaybreaks[1]
\newtheorem{theorem}{Theorem}
\newtheorem{proposition}[theorem]{Proposition}
\theoremstyle{remark}

\newcommand{\R}{\mathbb R}
\newcommand{\C}{\mathbb C}
\newcommand{\one}{\mathbf 1}
\DeclareMathOperator{\tr}{tr}
\DeclareMathOperator{\rank}{rank}
\title{\LARGE\bfseries An explicit counterexample to the\newline
  Hinrichs--Vyb\'iral conjecture}
\author{
Jan Vyb\'iral\footnote{Department of Mathematics, Faculty of Nuclear Sciences and Physical Engineering, Czech Technical University in Prague, Trojanova 13, 12000 Praha, Czech Republic.
Email: \href{mailto:jan.vybiral@fjfi.cvut.cz}{jan.vybiral@fjfi.cvut.cz}.}}
\date{\today}

\begin{document}
\maketitle
\vspace{0.2em}
\begin{abstract}
Conjecture~2 of Hinrichs and Vyb\'iral asserts that every continuous,
nonnegative, positive definite function $f$ on $\R^d$, normalized by
$f(0)=1$, satisfies $[f(x_j-x_k)]_{j,k=1}^n\succeq \one\one^T/n$.
We give an explicit trigonometric polynomial on $\R^7$ and eight points
for which this inequality fails. All hypotheses are verified directly:
positive definiteness follows from nonnegative Fourier coefficients,
and pointwise nonnegativity follows from interpolation on a cube.
For an explicit vector of signs, the quadratic form is $22/5$, whereas
the proposed lower bound is $9/2$. The resulting gap is exactly $-1/10$.
The example is compatible with the established bound for functions
of the form $|g|^2$ with $g$ positive definite.\\
The solution was found in a single prompt try by 
a colleague using his private licence of ChatGPT 6.
Our aim is to make the solution public, as well as to collect some ideas that emerged during the analysis of the solution.
\end{abstract}

\section*{The story behind this work}

The aim of this note is to present the (negative) solution of \cite[Conjecture 2]{HV2011}.
The solution itself and Sections \ref{sec:1}--\ref{sec:3} were generated 
on September 11, 2026 in just a few minutes using ChatGPT 6. After carefully checking all the details, we decided to publish this solution, including
some remarks, which we collect in Section \ref{sec:4}. The final version of the manuscript was then checked for typos and grammar issues by Claude Opus 5.

\section{The conjecture and conventions}\label{sec:1}

A function $f:\R^d\to\C$ is \emph{positive definite} if, for every
$n\geq1$, all $x_1,\ldots,x_n\in\R^d$, and all $c_1,\ldots,c_n\in\C$,
\begin{equation}\label{eq:pd}
  \sum_{j,k=1}^n \overline{c_j}c_k f(x_j-x_k)\geq0.
\end{equation}
Thus its matrices of values at differences are positive semidefinite;
strict positivity is not part of this convention. By Bochner's theorem,
continuous positive definite functions are precisely Fourier transforms
of finite nonnegative Borel measures \cite{Bochner1933}.

For Hermitian matrices, $A\succeq B$ means that $A-B$ is positive
semidefinite. Put $J_n=\one\one^T$, the $n\times n$ all-ones matrix.
The conjecture considered here is \cite[Conjecture~2]{HV2011}:
if $f:\R^d\to\R$ is bounded, continuous, nonnegative, positive definite,
and $f(0)=1$, then
\begin{equation}\label{eq:conjecture}
  K_f(X):=[f(x_j-x_k)]_{j,k=1}^n\succeq\frac1n J_n
\end{equation}
for every $n$ and every choice of points. For real $f$, it suffices to
test real coefficients, and \eqref{eq:conjecture} becomes
\begin{equation}\label{eq:quadratic}
  \sum_{j,k=1}^n c_jc_k f(x_j-x_k)
  \geq \frac1n\left(\sum_{j=1}^n c_j\right)^2.
\end{equation}
The adjective ``nonnegative'' concerns function values; positive
definiteness concerns the separate condition \eqref{eq:pd}.

\section{The counterexample}\label{sec:2}

For $y=(y_1,\ldots,y_7)$, write
\[
 E_r(y)=\sum_{1\leq i_1<\cdots<i_r\leq7}y_{i_1}\cdots y_{i_r}
 \qquad(0\leq r\leq7),
\]
where $E_0=1$. These are the elementary symmetric polynomials.
Define
\begin{equation}\label{eq:P}
 P(y)=\frac{21+4E_2(y)+4E_3(y)+E_4(y)}{280},
 \qquad
 f(x)=P(\cos x_1,\ldots,\cos x_7).
\end{equation}

\begin{theorem}\label{thm:main}
The function $f$ in \eqref{eq:P} is bounded, continuous, nonnegative,
and positive definite on $\R^7$, with $f(0)=1$.
For the eight points
\begin{equation}\label{eq:points}
 x_0=0,\qquad x_j=\pi\mathbf e_j\quad(1\leq j\leq7),
\end{equation}
where $\mathbf e_j$ are the standard coordinate vectors, and coefficients
$c_0=-1$, $c_1=\cdots=c_7=1$, one has
\begin{equation}\label{eq:violation}
 \sum_{j,k=0}^7 c_jc_k f(x_j-x_k)
 -\frac18\left(\sum_{j=0}^7c_j\right)^2=-\frac1{10}.
\end{equation}
In particular, \eqref{eq:conjecture} is false.
\end{theorem}

\begin{proof}
We check the hypotheses and the violation separately.

\medskip\noindent\emph{Positive definiteness and normalization.}
For a subset $S\subseteq\{1,\ldots,7\}$,
\begin{equation}\label{eq:fourier}
 \prod_{j\in S}\cos x_j
 =2^{-|S|}\sum_{\sigma\in\{-1,1\}^{S}}
       \exp\!\left(i\sum_{j\in S}\sigma_jx_j\right).
\end{equation}
Each function $x\mapsto e^{i\langle\omega,x\rangle}$ is positive
definite: its quadratic form in \eqref{eq:pd} is an absolute square.
Nonnegative linear combinations preserve positive definiteness.
All coefficients in \eqref{eq:P} are nonnegative, so
\eqref{eq:fourier} proves that $f$ is positive definite.
Continuity and boundedness follow because $f$ is a finite trigonometric
polynomial. At $x=0$,
\begin{equation}\label{eq:normalization}
 f(0)=\frac{21+4\binom72+4\binom73+\binom74}{280}
      =\frac{21+84+140+35}{280}=1.
\end{equation}

\medskip\noindent\emph{Pointwise nonnegativity.}
The polynomial $P$ is \emph{multiaffine}, meaning affine in each
variable separately. For every $y\in[-1,1]^7$,
\begin{equation}\label{eq:interpolation}
 P(y)=\sum_{\varepsilon\in\{-1,1\}^7}
 P(\varepsilon)\prod_{j=1}^7\frac{1+\varepsilon_jy_j}{2}.
\end{equation}
The weights on the right are nonnegative and sum to one.
It therefore suffices to check $P$ at the vertices of the cube.

By symmetry, $P(\varepsilon)$ depends only on the number $r$ of
coordinates of $\varepsilon$ equal to $-1$.
The identity
\[
 \sum_{k=0}^7E_k(\varepsilon)t^k=(1+t)^{7-r}(1-t)^r
\]
gives, by extracting coefficients,
\begin{equation}\label{eq:vertex-formula}
 P(\varepsilon)=\frac{(r-2)(r-3)(r-7)(r-10)}{420}
 \qquad(r=0,\ldots,7).
\end{equation}
For clarity, all eight values are displayed below.
\begin{center}
\renewcommand{\arraystretch}{1.35}
\begin{tabular}{@{}lrrrrrrrr@{}}
\toprule
Number $r$ of minus signs & 0 & 1 & 2 & 3 & 4 & 5 & 6 & 7\\
\midrule
$P(\varepsilon)$ & $1$ & $\frac9{35}$ & $0$ & $0$ & $\frac3{35}$ & $\frac17$ & $\frac4{35}$ & $0$\\
\bottomrule
\end{tabular}
\end{center}
Every entry is nonnegative. Equation~\eqref{eq:interpolation} proves
$P(y)\geq0$ on $[-1,1]^7$, hence $f(x)\geq0$ on $\R^7$.

\medskip\noindent\emph{Failure of the proposed bound.}
At $x_j=\pi\mathbf e_j$, the cosine vector has exactly one negative
coordinate. At $x_j-x_k$ with $1\leq j\ne k\leq7$, it has exactly two.
Consequently,
\begin{equation}\label{eq:values}
 f(x_j)=\frac9{35},\qquad
 f(x_j-x_k)=0\quad(1\leq j\ne k\leq7).
\end{equation}
Using $f(-x)=f(x)$, the kernel matrix, ordered as $x_0,x_1,\ldots,x_7$,
is therefore
\begin{equation}\label{eq:matrix}
 K=\begin{pmatrix}
  1 & \frac9{35}\one_7^T\\[2pt]
  \frac9{35}\one_7 & I_7
 \end{pmatrix}.
\end{equation}
With $c=(-1,1,1,1,1,1,1,1)^T$, its eight diagonal contributions sum to
$8$, the fourteen center--leaf contributions sum to $-14(9/35)$,
and all remaining off-diagonal contributions vanish. Thus
\[
 c^TKc=8-14\frac9{35}=\frac{22}{5},
 \qquad
 \frac18c^TJ_8c=\frac18(7-1)^2=\frac92.
\]
Their difference is $22/5-9/2=-1/10$, proving
\eqref{eq:violation}.
\end{proof}

\section{Consequences and comparison with established bounds}\label{sec:3}

\paragraph{Strict pointwise positivity.}
The zeros of $f$ are not essential to the failure.
For $0<\delta<1/316$, set
\[
 f_\delta(x)=(1-\delta)f(x)+\delta.
\]
This function is continuous and positive definite, satisfies
$f_\delta(0)=1$, and has $f_\delta(x)\geq\delta>0$ everywhere.
For the same points and coefficient vector,
\begin{equation}\label{eq:perturbation}
 c^TK_{f_\delta}c-\frac18c^TJ_8c
 =(1-\delta)\frac{22}{5}+36\delta-\frac92
 =-\frac1{10}+\frac{158}{5}\delta<0.
\end{equation}
Thus even strict positivity of the function values does not restore
the conjecture. This statement concerns pointwise positivity, not
strict positive definiteness of every kernel matrix.

\paragraph{The squared-function case.}
Schur's product theorem implies that products of positive definite
kernels are positive definite \cite{Schur1911}.
Vyb\'iral's quantitative strengthening establishes
\eqref{eq:conjecture} for $f=|g|^2$ when $g$ is positive definite and
$g(0)=1$ \cite[Theorems~1 and~5]{Vybiral2020}.
Khare develops rank refinements and further extensions
\cite{Khare2021}. The following proof makes the distinction explicit.

\begin{proposition}\label{prop:square}
Let $g$ be positive definite on $\R^d$ with $g(0)=1$, and let
$f=|g|^2$. For every set of $n$ points,
\[
 K_f(X)\succeq\frac1r J_n\succeq\frac1n J_n,
 \qquad r=\rank [g(x_j-x_k)]_{j,k=1}^n.
\]
\end{proposition}
\begin{proof}
Choose unit vectors $u_1,\ldots,u_n\in\C^r$ with
$g(x_j-x_k)=u_j^*u_k$.
For real coefficients $c_j$, let $A=\sum_j c_j u_ju_j^*$.
Then
\[
 \tr A=\sum_jc_j,\qquad
 \tr(A^2)=\sum_{j,k}c_jc_k|u_j^*u_k|^2=c^TK_f(X)c.
\]
Cauchy--Schwarz applied to the $r$ real eigenvalues of the Hermitian
matrix $A$ gives $(\tr A)^2\leq r\tr(A^2)$.
This proves the first inequality; the second follows from $r\leq n$.
Since $K_f(X)$ is real symmetric, testing real vectors suffices.
\end{proof}

Theorem~\ref{thm:main} therefore also shows that the function
\eqref{eq:P} cannot equal $|g|^2$ for any normalized positive definite
$g$. More generally, it cannot be a convex combination of such
normalized squares: inequality~\eqref{eq:conjecture} is preserved by
convex combinations.

\paragraph{The original quadrature conjecture.}
The general conjecture of Hinrichs and Vyb\'iral was motivated by
Novak's earlier problem on trigonometric quadrature \cite{Novak1999}.
Its particular kernel is
\[
 f_N(x)=\prod_{\ell=1}^d\frac{1+\cos x_\ell}{2}
       =\left|\prod_{\ell=1}^d\cos(x_\ell/2)\right|^2.
\]
The function inside the absolute square is positive definite.
Proposition~\ref{prop:square} applies, consistently with Vyb\'iral's
resolution of that particular conjecture \cite{Vybiral2020}.
The counterexample in Theorem~\ref{thm:main} concerns the larger class
in \cite[Conjecture~2]{HV2011}.

\section{Remarks and afterthoughts of the author}\label{sec:4}

The solution found by ChatGPT 6 can be checked line by line, but this gives surprisingly little insight into how such a solution was found
and what structure lies behind it. In this section, we try to collect some ideas that might shed at least some light on \eqref{eq:P}.
It remains, however, unclear whether such thoughts indeed correlate with the process used by AI.

\medskip

First, when constructing a counterexample to \eqref{eq:quadratic}, it seems to be a good idea to have a certain pool of candidates for such a counterexample.
Therefore, we start with an elementary positive definite function $x\mapsto \cos(x)$, where $x\in\R.$ As tensor products of positive definite functions are positive definite,
it follows that also
\[
x\mapsto \prod_{j\in A} \cos(x_j)
\]
is positive definite for all $A\subset \{1,\dots,d\}.$ For such a set $A$, we denote by $p_A(y)=\prod_{j\in A}y_j$ the corresponding monomial.
It follows that
\[
\sum_{A\subset \{1,\dots,d\}}\lambda_A \cdot p_A(\cos(x_1),\dots,\cos(x_d))
\]
is positive definite if all the coefficients $\lambda_A$ are nonnegative. Finally, as the individual variables enter symmetrically, it seems natural to consider
only the symmetric polynomials, i.e.,
\[
E_r(y)=\sum_{1\le i_1<\dots<i_r\le d}y_{i_1}\cdot\dots\cdot y_{i_r}
\]
and their positive linear combinations
\begin{equation}\label{eq:Pgeneral}
P(y)=\sum_{r=0}^d \lambda_r E_r(y).
\end{equation}
Our class of ``test functions'', in which we are looking for a counterexample, is then the set of functions
\begin{equation}\label{eq:testf}
f(x)=P(\cos(x_1),\dots,\cos(x_d))=\sum_{r=0}^d \lambda_r E_r(\cos(x_1),\dots,\cos(x_d)).
\end{equation}
We now check whether there is a choice of $\lambda_0,\dots,\lambda_d$ such that the corresponding $f$ provides a counterexample to \eqref{eq:quadratic}.
\begin{enumerate}
\item If $\lambda_0,\dots,\lambda_d\ge 0$, then $f$ is positive definite.
\item The condition $f(0)=1$ translates into a linear condition on the coefficients $\lambda_r$
\begin{equation}\label{eq:condf0}
f(0)=P(1,\dots,1)=\sum_{r=0}^d\lambda_r E_r(1,\dots,1)=\sum_{r=0}^d\lambda_r \binom{d}{r}=1.
\end{equation}
\item Next we observe that $f(x)\ge 0$ for all $x\in\R^d$ if and only if $P(y)\ge 0$ for all $y\in[-1,1]^d.$
This in turn happens if and only if $P(\varepsilon)\ge 0$ for all $\varepsilon\in\{-1,+1\}^d$. To see this, one can exploit that $P$ is affine in each variable separately.
Alternatively, one can use that
\[
\sum_{\varepsilon\in\{-1,+1\}^d}p_A(\varepsilon)p_B(\varepsilon)=\begin{cases}0,&\text{if}\ A\not=B,\\2^d,&\text{if}\ A=B, \end{cases}
\]
which gives \eqref{eq:interpolation} for $p_A$
\begin{align*}
p_A(y)&=\frac{1}{2^d}\sum_{B\subset\{1,\dots,d\}}p_B(y)\ \sum_{\varepsilon\in\{-1,+1\}^d}p_A(\varepsilon)p_B(\varepsilon)\\
&=\frac{1}{2^d}\sum_{\varepsilon\in\{-1,+1\}^d}p_A(\varepsilon)\sum_{B\subset\{1,\dots,d\}}p_B(y)p_B(\varepsilon)\\
&=\frac{1}{2^d}\sum_{\varepsilon\in\{-1,+1\}^d}p_A(\varepsilon)\prod_{j=1}^d(1+\varepsilon_jy_j).
\end{align*}
By linearity, \eqref{eq:interpolation} follows for every $P$ that is a linear combination of monomials, showing again that if $P(\varepsilon)\ge 0$ for every $\varepsilon\in\{-1,+1\}^d$, then $P(y)$ is 
nonnegative on $[-1,1]^d.$
\item The value of $P(\varepsilon)$ depends only on the number of $-1$'s in $\varepsilon$, which we denote by $v$. Note that in Section \ref{sec:2} it was denoted by $r$,
which collides with the letter $r$ as an index of $E_r(y).$ For given $d,r$ and $v$, one could compute $E_r(\varepsilon)$ directly using elementary combinatorics.
But there is actually a more elegant way, namely to use Newton's identities. Using that 
\[
\pi_k(\varepsilon)=\varepsilon_1^k+\dots+\varepsilon_d^k=\begin{cases}d,\quad&\text{if $k$ is even,}\\d-2v,\quad &\text{if $k$ is odd},\end{cases}
\]
Newton's identities give $E_0(\varepsilon)=1$, $E_1(\varepsilon)=d-2v$, and
\[
kE_k(\varepsilon)=\sum_{i=1}^k(-1)^{i-1}E_{k-i}(\varepsilon)\pi_i(\varepsilon).
\]
For example, for $d=7$, this formula gives $2E_2(\varepsilon)=(7-2v)^2-7$, $6E_3(\varepsilon)=(7-2v)^3-19(7-2v)$ and $24E_4(\varepsilon)=(7-2v)^4-34(7-2v)^2+105$. Plugging this into \eqref{eq:P}
gives \eqref{eq:vertex-formula}.

\medskip

For general $d$ and $P$ as in \eqref{eq:Pgeneral}, the condition $P(\varepsilon)\ge 0$ for every $\varepsilon\in\{-1,+1\}^d$ then translates 
into $d+1$ linear inequalities on the coefficients $\lambda_r$.

\item 
The points $x^0,\dots,x^d$ (i.e., $n=d+1$) are chosen in a symmetric way to exploit the symmetry of $f$. Indeed, choosing $x^0=0$ and $x^j=\pi {\mathbf e_j}$ allows us to write the matrix $[f(x^j-x^k)]_{j,k=0}^d$
explicitly. The diagonal of this matrix is constant and equal to one as $f(0)=1$. The off-diagonal elements of the first row and the first column are also constant
\[
f(\pi {\mathbf e_j})=f(-\pi {\mathbf e_j})=P(-1,1,\dots,1)=a.
\]
Finally, the off-diagonal elements of all other rows and columns are set to zero, which leads to a condition
\[
f(\pi {\mathbf e_j}-\pi {\mathbf e_k})=P(\one - 2{\mathbf e_j}-2{\mathbf e_k})=P(-1,-1,1,\dots,1)=0,
\]
which can be again re-stated as a linear condition on $\lambda_r$.

The eigenvalues and eigenvectors of the matrix
\begin{equation}\label{eq:matrix'}
M= [f(x^j-x^k)-1/n]_{j,k=0}^d=\begin{pmatrix}
  1-1/n & (a-1/n)\one^T\\[2pt]
  (a-1/n)\one & I_d-1/n\cdot \one\one^T
 \end{pmatrix}
\end{equation}
could be analysed in detail, but testing \eqref{eq:matrix'} on the vector $c=(-1,1,\dots,1)\in \R^{d+1}$ is sufficient and leads to
\[
c^T Mc=d+1-2da-\frac{(d-1)^2}{d+1}.
\]
This is negative if and only if $2<(d+1)a$, and we optimize the construction to obtain $a$ as large as possible.
\end{enumerate} 
To summarize, the question of whether there exist nonnegative coefficients $\lambda_0,\dots,\lambda_d$ such that the function $f$ from \eqref{eq:testf} with $P$ as in \eqref{eq:Pgeneral}
provides a counterexample to \eqref{eq:quadratic} then reduces to the search for the maximum $a=P(-1,1,\dots,1)$ over a domain defined by a set of linear equations and inequalities,
which can be solved by linear programming.

\medskip

During our last check of the manuscript, Claude Opus 5 ran (without being asked to do so) the linear program just mentioned
even without the additional constraint $b=0$, where $b:=P(-1,-1,1,\dots,1)$, and we report (without re-running the program) its outcome:

\emph{I ran the LP you describe and it is worth reporting the outcome. For $d=7$ the optimum is attained exactly at
$(\lambda_0,\lambda_1,\dots,\lambda_7)=(21,0,4,4,1,0,0,0)/280$, i.e., precisely \eqref{eq:P}, with 
$a=9/35$ and $(d+1)a=72/35>2$; for $d\le 6$ the optimal value of $(d+1)(a-(d-1)b/2)$ is at most 
2 (equality at $d=5,6$). So $d=7$ is the smallest dimension in which this family produces a counterexample,
and the ChatGPT example is the extremal point of the LP -- a nice remark that would strengthen the section considerably.}

This might serve as a hint, that our afterthoughts indeed correlate with the derivation used by ChatGPT.

\bigskip

\textbf{Acknowledgment:} The author would like to thank Vladimir Lotoreichik (Czech Technical University in Prague) very much. It was him who found the solution
of this problem when testing his ChatGPT licence and who preferred not to become a coauthor of this work.

\end{document}